\documentclass{amsart}

\usepackage{amssymb}
\usepackage{amsmath}
\usepackage{amsthm}

\newtheorem{thm}{Theorem}[section]

\newtheorem{cor}[thm]{Corollary}

\theoremstyle{definition}
\newtheorem{defn}[thm]{Definition}

\theoremstyle{remark}

\newcommand{\system}[1]{\mbox{\fontfamily{cmss}\fontshape{n}\fontseries{m}%
    \selectfont#1}}
\newcommand{\RCA}{\system{RCA}\ensuremath{_0}}

\newcommand{\RT}{\system{RT}}
\newcommand{\ACA}{\system{ACA}\ensuremath{_0}}

\newcommand{\SRT}{\system{SRT}}
\newcommand{\HT}{\system{HT}}
\newcommand{\D}{\system{D}}

\begin{document}

\title{Effectiveness of Hindman's Theorem for bounded sums}

\author[D.\ D.\ Dzhafarov]{Damir D. Dzhafarov}
\address{Department of Mathematics\\
University of Connecticut\\
196 Auditorium Road\\ Storrs, Connecticut 06269 U.S.A.}
\email{damir@math.uconn.edu}

\author[C.\ G.\ Jockusch]{Carl G. Jockusch, Jr.}
\address{Department of Mathematics\\
University of Illinois\\
1409 W.\ Green Street, 
Urbana, Illinois 61801, U.S.A.}
\email{jockusch@math.uiuc.edu}

\author[R.\ Solomon]{Reed Solomon}
\address{Department of Mathematics\\
University of Connecticut\\
196 Auditorium Road\\ Storrs, Connecticut 06269 U.S.A.}
\email{david.solomon@uconn.edu}

\author[L.\ B.\ Westrick]{Linda Brown Westrick}
\address{Department of Mathematics\\
University of Connecticut\\
196 Auditorium Road\\ Storrs, Connecticut 06269 U.S.A.}
\email{linda.westrick@uconn.edu}

\thanks{Dzhafarov was partially supported by NSF grant DMS-1400267. Jockusch thanks his coauthors for their hospitality during a visit to the University of Connecticut in May, 2015, during which the work in this paper was largely done. His visit was supported by the University of Connecticut Department of Mathematics.}



\begin{abstract}
  We consider the strength and effective content of restricted
  versions of Hindman's Theorem in which the number of colors is
  specified and the length of the sums has a specified finite bound.
  Let $\HT^{\leq n}_k$ denote the assertion that for each $k$-coloring
  $c$ of $\mathbb{N}$ there is an infinite set $X \subseteq
  \mathbb{N}$ such that all sums $\sum_{x \in F} x$ for $F \subseteq
  X$ and $0 < |F| \leq n$ have the same color.  We prove that there is
  a computable $2$-coloring $c$ of $\mathbb{N}$ such that there is no
  infinite computable set $X$ such that all nonempty sums of at most
  $2$ elements of $X$ have the same color.  It follows that $\HT^{\leq
    2}_2$ is not provable in $\RCA$ and in fact we show that it
  implies $\SRT^2_2$ in $\RCA$.  We also show that there is a
  computable instance of $\HT^{\leq 3}_3$ with all solutions computing
  $0'$.  The proof of this result shows that $\HT^{\leq 3}_3$ implies
  $\ACA$ in $\RCA$.
\end{abstract}

\maketitle

\section{Introduction}

Hindman's Theorem (denoted $\HT$) asserts that for every coloring of
$\mathbb{N}$ with finitely many colors there is an infinite set $X
\subseteq \mathbb{N}$ such that all nonempty finite sums of distinct
elements of $X$ have the same color.  Hindman's Theorem was proved by
Neil Hindman \cite{H}.  Hindman's original proof was a complicated
combinatorial argument, and simpler proofs have been subsequently
found.  These include combinatorial proofs by Baumgartner \cite{B} and
by Towsner \cite{T} and a proof using ultrafilters by Galvin and
Glazer (see \cite{C}).

We assume that the reader is familiar with the basic concepts of
computability theory and of reverse mathematics.  For information on
these topics see, respectively, the books by Soare \cite{So} and
Simpson \cite{Si}.  Our notation is standard.  In particular, let
$\mathbb{N}$ be the set of positive integers, and for $k \in
\mathbb{N}$ we identify $k$ and $\{0, 1, \dots, k - 1\}$.  A
$k$-\emph{coloring} of $\mathbb{N}$ is a function $c : \mathbb{N}
\rightarrow k$.  A set $Z \subseteq \mathbb{N}$ is
\emph{monochromatic} for a coloring $c$ if $c(x) = c(y)$ for all $x, y
\in Z$.

The effective content of Hindman's Theorem and its strength as a
sentence of second-order arithmetic were studied by Blass, Hirst, and
Simpson \cite{BHS}.  They showed that every computable instance $c$ of
$\HT$ has a solution $X$ computable from $0^{(\omega +2)}$ and,
correspondingly, that $\HT$ is provable in the system ACA$_0^+$
obtained by adding to $\RCA$ the statement $(\forall X) [X^{(\omega)}
\text{ exists}]$.  In the other direction, they showed that there is a
computable instance $c$ of $\HT$ such that all solutions $X$ compute
$0'$ and, correspondingly, that $\HT$ implies $\ACA$ in ${\RCA}$.

There is obviously a significant gap between the upper and lower
bounds given in the previous paragraph, and closing these gaps has been
a major issue in reverse mathematics.  In particular it is
not known whether there is an $n$ such that every computable instance
of Hindman's Theorem has a $\Sigma^0_n$ solution and, correspondingly,
whether $\HT$ is provable from $\ACA$ in $\RCA$.

In the current paper we study the strength and effective content of
Hindman's Theorem when it is restricted to sums of bounded length.
One might think that such restricted versions of Hindman's Theorem are
far weaker than Hindman's Theorem itself, but in fact it is unknown
whether this is true.  In fact it is a major open problem in
combinatorics (see \cite{HLS}, Question 12) whether every proof of
Hindman's Theorem for sums of length at most two also proves Hindman's
Theorem.  We now state these bounded versions formally.

\begin{defn}
  For a finite nonempty set $F \subseteq \mathbb{N}$, we let $\sum F$
  denote the sum of the elements of $F$. For $X \subseteq \mathbb{N}$
  and $n \geq 1$, we define
\[
\text{FS}^{\leq n}(X) =  \left\{ \sum F \mid F \subseteq X \text{ and } 1 \leq |F| \leq n \right\}.
\] 
\end{defn}

\begin{defn}
  Let $\HT^{\leq n}_k$ denote the statement that for every coloring $c:
  \mathbb{N} \rightarrow k$, there is an infinite set $X$ such that
  $\text{FS}^{\leq n}(X)$ is monochromatic.
\end{defn}

We show in Section \ref{22} that for every $\Delta^0_2$ set $X$ there
is a computable instance $c$ of $\HT^{\leq 2}_2$ such that every
solution $H$ to $c$ computes an infinite subset of $X$ or
$\overline{X}$.  It follows that $\HT^{\leq 2}_2$ has a computable
instance with no computable solution and hence is not provable in
$\RCA$.  In fact, our proof shows that $\HT^{\leq 2}_2$ implies $\SRT^2_2$
(Stable Ramsey's Theorem for $2$-colorings of pairs) in $\RCA$.   Next we show in Section
\ref{33} that there is a computable instance of $\HT^{\leq 3}_3$ such that
every solution computes $0'$ and, correspondingly, that $\HT^{\leq 3}_3$
implies $\ACA$ in $\RCA$.  Our proof uses a very ingenious trick from
Blass, Hirst, and Simpson \cite{BHS}, combined with some new ideas.

The final section lists many open questions.

\section{Hindman's Theorem for sums of length at most $2$}    \label{22}

Our first theorem concerns $\HT^{\leq 2}_ 2$ and implies that it has
a computable instance $c$ with no computable solution $X$.  

\begin{thm}
  Let $A$ be a $\Delta^0_2$ set. There is a computable coloring
  $c:\mathbb{N} \rightarrow 2$ such that if $W$ is an infinite set
  with $\text{FS}^{\leq 2}(W)$ monochromatic, then there is an
  infinite set $Y \leq_T W$ such that $Y \subseteq A$ or $Y \subseteq
  \overline{A}$.
\end{thm}

\begin{proof}
  Fix a $\Delta^0_2$ set $A$ and a computable $\{0,1\}$-valued function $f(k,s)$ such
  that $A(k) = \lim_s f(k,s)$. For $k \geq 0$ and $i \in \{ 1,2 \}$,
  define
\[
\mathcal{O}_{k,i} = \{ s \in \mathbb{N} \mid s \equiv i \cdot 3^k \, \text{mod} \, 3^{k+1} \}.
\]
If $s$ is written as $s = i_0 \cdot 3^{k_0} + \cdots + i_m \cdot
3^{k_m}$ with $k_0 < \cdots < k_m$ and each $i_j \in \{ 1,2 \}$, then
$s \in \mathcal{O}_{k,i}$ if and only if $k=k_0$ and $i = i_0$.  The
sets $\mathcal{O}_{k,i}$ give a computable partition of $\mathbb{N}$
such that if $s,t \in \mathcal{O}_{k,1}$, then $s+t \in
\mathcal{O}_{k,2}$ and if $s,t \in \mathcal{O}_{k,2}$, then $s+t \in
\mathcal{O}_{k,1}$. Furthermore, if $s \in \mathcal{O}_{k,i}$ and $t
\in \mathcal{O}_{k',i'}$ with $k < k'$ and $i' \in \{1,2\}$, then $s+t \in
\mathcal{O}_{k,i}$. For any $s \in \mathbb{N}$, we let $k_s, i_s$
be the unique numbers $k, i$ such that $s \in \mathcal{O}_{k,i}$. We
define our coloring $c$ by
\[
c(s) = 
\begin{cases}
f(k_s,s) & \text{if } i_s = 1, \\
1-f(k_s,s) & \text{if } i_s = 2.
\end{cases}
\]
The first important property of this coloring is that for each $k$ we have
$c(s) \neq c(t)$ whenever $s  \in \mathcal{O}_{k,1}$ and $t \in \mathcal{O}_{k,2}$ are both sufficiently large.    This holds
since for  sufficiently large $s \in \mathcal{O}_{k,1}$ and $t \in \mathcal{O}_{k,2}$  we have $c(s) = f(k, s) = A(k)$ and $c(t) = 1 - f(k , t) = 1 - A(k)$.   It follows that for any monochromatic set $Z$, either
$Z \cap \mathcal{O}_{k,1}$ is finite or $Z \cap \mathcal{O}_{k,2}$ is finite.

Fix an infinite set $W$ with $\text{FS}^{\leq 2}(W)$ monochromatic.    We claim that $W \cap \mathcal{O}_{k,i}$ is finite for each $k \in \mathbb{N}$ and $i \in \{1,2\}$.    Suppose first that $W \cap \mathcal{O}_{k,1}$ is infinite.    Let $S$ be the set of all sums $a + b$ where $a, b$ are distinct elements of
$W \cap \mathcal{O}_{k,1}$.   Then $S$ is infinite and $S \subseteq \mathcal{O}_{k,2} \cap \text{FS}^{\leq 2}(W)$.    Let $Z =
W \cup S$.   Then $Z$ is monochromatic since $Z \subseteq \text{FS}^{\leq 2}(W)$.    Furthermore,
$Z \cap \mathcal{O}_{k,1}$ and $Z \cap \mathcal{O}_{k,2}$ are both infinite, contradicting the previous paragraph.   This shows that $W \cap \mathcal{O}_{k,1}$ is finite, and the proof that $W \cap \mathcal{O}_{k,2}$ is finite is analogous.  It follows
that there are infinitely many $k$ such that $W \cap (\mathcal{O}_{k,1} \cup \mathcal{O}_{k,2})$ is nonempty.   We call such numbers $k$ \emph{informative} since, as the next claim shows, $W$ can compute $A(k)$ for all informative $k$.

We claim that if $s \in W \cap \mathcal{O}_{k,i}$ then
\[
A(k) = 
\begin{cases}  c(s) & \text{if }  i = 1, \\
1-c(s) & \text{if } i = 2.
\end{cases}
\]
To prove the above claim, fix $s \in W \cap \mathcal{O}_{k,i}$.   Note that
$\text{FS}^{\leq 2}(W) \cap \mathcal{O}_{k, i}$ is infinite, since it contains all
sums $s+b$ with  $b \in W \cap \mathcal{O}_{k', i'}$ for some $k' > k$, and 
$i' \in \{1,2\}$, and there are infinitely many such $b$.
Let $t$ be an element of $\text{FS}^{\leq 2}(W) \cap \mathcal{O}_{k, i}$  sufficiently large that
$f(k,t) = A(k)$.   Since $\text{FS}^{\leq 2}(W)$  is monochromatic,
$c(s) = c(t)$.   Hence  $c(s) = c(t) = f(k,t) = A(k)$ if $i = 1$,   and $c(s) = c(t) = 1 - f(k,t) = 1 - A(k)$
if $i = 2$.   The claim is proved.

For $i \in \{0,1\}$ let $B_i$ be the set of numbers $k$ such that $W$ can compute that $A(k) = i$.  More precisely, define
\[
B_i = \{k \mid (\exists n)[(n \in W \cap \mathcal{O}_{k,1} \ \& \ c(n) = i) \text{ or } (n \in \mathcal{O}_{k,2}
\ \& \ c(n) = 1-i)]\}
\]

By the above claim, $B_1 \subseteq A$ and $B_0 \subseteq \overline{A}$.   Also, each set $B_i$ is c.e. in $W$.   Finally, if $k$ is informative, then $k \in B_0 \cup B_1$.  Since there are infinitely many informative numbers, $B_0 \cup B_1$ is infinite, and so $B_0$ or $B_1$ is infinite.   Fix $i$ such that $B_i$ is infinite, and let $Y$ be an infinite
$W$-computable subset of $B_i$.    Then $Y$ is the desired infinite $W$-computable subset of
$A$ or $\overline{A}$.
 \end{proof}

The next corollary follows by taking $A$ to be a bi-immune $\Delta^0_2$ 
set, for example a $\Delta^0_2$ $1$-generic set.

\begin{cor}
  There is a computable coloring $c:\mathbb{N} \rightarrow 2$ such
  that if $X$ is an infinite computable set, then $\text{FS}^{\leq
    2}(X)$ is not monochromatic.
\end{cor}

The next corollary follows immediately.

\begin{cor}
    $\HT^{\leq 2}_2$ is not provable in $\RCA$.
\end{cor}

We now sharpen the previous corollary.  Let $\D^2_2$ be the assertion
that for every $\{0,1\}$-valued function $f(x,s)$ such that for all
$x$, $\lim_s f(x,s)$ exists there is an infinite set $G$ and $j < 2$
such that $\lim_s f(x,s) = j$ for all $x \in G$.  (The principle
$\D^2_2$ was defined in \cite{CJS}, Section 7.)  

\begin{cor}
$\RCA \vdash \HT^{\leq 2}_2 \rightarrow \D^2_2$.
\end{cor}

The proof follows by formalizing the proof of the theorem and the
proof of the Limit Lemma.

It was shown by Chong, Lempp and Yang (\cite{CLY}, Theorem 1.4) that $D^2_2$ implies
 $\Sigma^0_2$-bounding (B$\Sigma^0_2$) in \RCA, and hence (justifying a hidden
use of B$\Sigma^0_2$ in the proof of \cite{CJS}, Lemma 7.10), $\D^2_2$
is equivalent to Stable Ramsey's Theorem for Pairs $\SRT^2_2$ as
defined in Statement 7.5 of \cite{CJS}.

\begin{cor}
$\RCA \vdash  \HT^{\leq 2}_2 \rightarrow\SRT^2_2$.
\end{cor}

\section{Hindman's Theorem for sums of length at most $3$}  \label{33}

We now strengthen the results of the previous section, at the cost of
allowing longer sums and more colors.  We start by considering
$\HT^{\leq 3}_4$ and then improve the results to $\HT^{\leq 3}_3$.

\begin{thm}
  There is a computable coloring $c:\mathbb{N} \rightarrow 4$ such
  that if $X$ is infinite with $\text{FS}^{\leq 3}(X)$ monochromatic,
  then $0' \leq_T X$.
\end{thm}

\begin{proof}
Let $f:\mathbb{N} \rightarrow \mathbb{N}$ be a computable 1-1 function. We will define a computable coloring $c:\mathbb{N} \rightarrow 4$ such that 
if $X$ is infinite with $\text{FS}^{\leq 3}(X)$ monochromatic, then $X$ computes $\text{range}(f)$. 

For $n \in \mathbb{N}$, write $n = i_0 \cdot 3^{k_0} + \cdots + i_\ell \cdot 3^{k_\ell}$ with $k_0 < \cdots < k_\ell$ and each $i_j \in \{ 1,2 \}$. Define 
$\lambda(n) = k_0$, $\mu(n) = k_\ell$ and $i(n) = i_0$. We will use several properties of the functions $\lambda(n)$, $\mu(n)$ and $i(n)$. The following are 
all straightforward to establish.
\begin{enumerate}
\item[(P1)] If $\lambda(n) < \lambda(m)$, then $\lambda(n+m) = \lambda(n)$ and $i(n+m) = i(n)$. 
\item[(P2)] If $\lambda(n) = \lambda(m)$ and $i(n) = i(m) = 1$, then $\lambda(n+m) = \lambda(n)$ and $i(n+m) = 2$. 
\item[(P3)] If $\lambda(n) = \lambda(m)$ and $i(n) = i(m) = 2$, then $\lambda(n+m) = \lambda(n)$ and $i(n+m) = 1$.
\item[(P4)] If $\mu(n) < \lambda(m)$, then $\lambda(n+m) = \lambda(n)$ and $\mu(n+m) = \mu(m)$.
\end{enumerate}

For $n = i_0 \cdot 3^{k_0} + \cdots + i_\ell \cdot 3^{k_\ell}$ with
the $i_j$ and $k_j$ as above, we refer to the intervals $(k_j,
k_{j+1})$ for $j < \ell$ as the \textit{gaps of} $n$. A gap $(a,b)$ of
$n$ is a \textit{short gap in} $n$ if there is a $y \leq a$ such that
$y \in \text{range}(f)$ but there is no $x \leq b$ such that $f(x) =
y$. (Note that whether a gap $(a,b)$ in $n$ is short does not depend
on $n$.) A gap $(a,b)$ of $n$ is a \textit{very short gap in} $n$ if
there is a $y \leq a$ for which there is an $x \leq \mu(n)$ with
$f(x)=y$ but no $x \leq b$ for which $f(x) = y$.  Note that we can
computably determine the very short gaps in $n$ but can only
computably enumerate the short gaps in $n$.

For each $n$, we let $\text{SG}(n)$ be the number of short gaps in $n$
and we let $\text{VSG}(n)$ be the number of very short gaps in $n$. As
above, we can compute VSG$(n)$ but in general can only approximate
SG$(n)$ in an increasing fashion as we discover the short gaps. We
define our computable coloring by
\[
c(n) = 
\begin{cases}
\text{VSG}(n) \, \text{mod} \, 2 & \text{if } i(n) = 1, \\
2+ (\text{VSG}(n) \, \text{mod} \, 2) & \text{if } i(n) = 2.
\end{cases}
\]

Let $X$ be an infinite set such that $\text{FS}^{\leq 3}(X)$ is
monochromatic. We establish the following two properties.
\begin{enumerate}
\item[(P5)] For all $n, m \in X$, $i(n) = i(m)$.
\item[(P6)] For $k \geq 0$, there is at most one $n \in X$ such that
  $\lambda(n) = k$.
\end{enumerate}
(P5) holds because $i(n) = 1$ implies $c(n) \in \{ 0,1 \}$ and $i(m) =
2$ implies $c(m) \in \{ 2,3 \}$.  (P6) holds since if $n \neq m \in X$
with $\lambda(n) = \lambda(m)$ (and by (P5), $i(n) = i(m)$), then by
(P2) and (P3), $i(n+m) \neq i(n)$ contradicting (P5).

By (P6), we can assume without loss of generality (by computably
thinning out $X$) that if $n, m \in X$ with $n < m$, then $\mu(n) <
\lambda(m)$. The argument now proceeds almost identically to the proof
of Theorem 2.2 in Blass, Hirst and Simpson with one minor change.

First, we claim that for all $n \in \text{FS}^{\leq 2}(X)$,
$\text{SG}(n)$ is even. For this claim, it is important that $n$ is a
sum of at most two elements of $X$. In particular, this claim need not
hold for an arbitrary element of $\text{FS}^{\leq 3}(X)$.

Fix $m \in X$ such that $n < m$, $\mu(n) < \lambda(m)$ and for all $y
\leq \mu(n)$, if $y \in \text{range}(f)$, then there is an $x \leq
\lambda(m)$ with $f(x) = y$. Since $n$ is a sum of at most two
elements of $X$, $n+m \in \text{FS}^{\leq 3}(X)$.  Because $\mu(n) <
\lambda(m)$, the gaps in $n+m$ consist of the gaps in $n$, the gaps in
$m$, and the gap $(\mu(n), \lambda(m))$. We want to count the number of
very short gaps in $n+m$. By the choice of $m$, the gap $(\mu(n),
\lambda(m))$ is not very short in $n+m$. By (P4), $\mu(n+m) = \mu(m)$,
so each gap in $m$ is very short in $n+m$ if and only if it is very
short in $m$. Finally, if $(a,b)$ is a gap in $n$, then $b \leq
\mu(n)$ and hence by the choice of $m$, $(a,b)$ is very short in $n+m$
if and only if it is short in $n$. Therefore, we have
\[
\text{VSG}(n+m) = \text{SG}(n) + \text{VSG}(m).
\]
Since $c(m) = c(n+m)$, the parity of $\text{VSG}(m)$ is equal to the parity of $\text{VSG}(n+m)$ and therefore $\text{SG}(n)$ is even.   

The last claim we need is that if $n, m \in X$ with $n < m$, then for
all $y \leq \mu(n)$, $y \in \text{range}(f)$ if and only if there is
an $x \leq \lambda(m)$ with $f(x) = y$. Note that this claim gives us
a method to compute $\text{range}(f)$ from $X$, completing the
proof. To prove the claim, suppose for a contradiction that there is a
$y \leq \mu(n)$ such that $y \in \text{range}(f)$ but there is no $x
\leq \lambda(m)$ with $f(x) = y$. In this case, the gap $(\mu(n),
\lambda(m))$ is short in $n+m$. Therefore, because the gaps of $n$
(respectively $m$) are short in $n+m$ if and only if they are short in
$n$ (respectively $m$), we have
\[
\text{SG}(n+m) = \text{SG}(n) + \text{SG}(m) + 1.
\]
Since $n \neq m \in X$, we have $n+m \in \text{FS}^{\leq 2}(X)$ and
hence $\text{SG}(n)$, $\text{SG}(m)$ and $\text{SG}(n+m)$ are all even,
giving the desired contradiction.
\end{proof}

Formalizing the proof of this theorem in $\RCA$, we obtain the following corollary.

\begin{cor}
$\RCA \vdash \HT^{\leq 3}_4 \rightarrow \ACA$. 
\end{cor}

We now improve the previous theorem and corollary from $4$ colors to
$3$ colors.

\begin{thm}
  There is a computable coloring $c:\mathbb{N} \rightarrow 3$ such
  that if $X$ is infinite with $\text{FS}^{\leq 3}(X)$ monochromatic,
  then $0' \leq_T X$.
\end{thm}

\begin{proof}
For any $k$ and $i \in \{1,2,3,4,5,6\}$, let $\mathcal{O}_{k,i} = \{ n : n
\equiv i\cdot 7^k \mod 7^{k+1}\}$.  Let $i_n$ denote the first nonzero
heptary bit of $n$, which occurs in the $k_n$th place, so that $n \in
\mathcal{O}_{k_n, i_n}$.  Color each $n \in \mathbb{N}$ red, green or blue as follows with the slash indicating a choice between
two colors depending on whether $\text{VSG}(n)$ is even or odd.
$$c(n) = \begin{cases} R/G &\text{ if } \text{VSG}(n) \text{ is even/odd and } i_n \equiv \pm 1 \mod 7, \\
  G/B &\text{ if } \text{VSG}(n) \text{ is even/odd and } i_n \equiv \pm 2 \mod 7, \\
  B/R &\text{ if } \text{VSG}(n) \text{ is even/odd and } i_n \equiv \pm 3 \mod 7. \\
				\end{cases}$$
				
Let $X\subseteq \mathbb{N}$ be an infinite set such that $\text{FS}^{\leq3}(X)$ is monochromatic. We claim that $X\cap \mathcal{O}_{k,i}$ cannot contain more than $2$ elements.  
To prove this claim, assume that $x,y,z$ are distinct elements of  $X \cap \mathcal{O}_{k,i}$ and hence 
$x+y \in \mathcal{O}_{k,(2i \mod 7)} \cap \text{FS}^{\leq3}(X)$ and $x+y+z \in \mathcal{O}_{k, (3i\mod 7)} \cap \text{FS}^{\leq3}(X)$.  Consider the following table of multiplication facts.
$$\begin{array}{c|c|c}
i & 2i \mod 7 & 3i\mod 7 \\\hline
\pm 1 & \pm 2 & \pm 3 \\
\pm 2 & \pm 3 & \pm 1 \\
\pm 3 & \pm 1 & \pm 2 \\
\end{array}$$
The table shows that $\text{FS}^{\leq3}(X)$ must contain elements from each of the sets $\mathcal{O}_{k, \pm1 \mod 7}, \mathcal{O}_{k, \pm2 \mod 7},$ and 
$\mathcal{O}_{k, \pm3 \mod 7}$ (where $\mathcal{O}_{k, \pm 1 \mod 7} = \mathcal{O}_{k,1} \cup \mathcal{O}_{k,6}$ and similarly for the other sets). 
However, by the definition of the coloring $c$, it is not possible for a monochromatic set to intersect all three of these sets. Therefore, if $x,y,z \in X \cap \mathcal{O}_{k,i}$ are distinct, 
then $\text{FS}^{\leq3}(X)$ is not monochromatic, proving the claim.  

By the claim, if $\text{FS}^{\leq3}(X)$ is monochromatic, then $X$ must include elements $n$ for which $k_n$ is arbitrarily large. Also, we can computably thin $X$ so that all of its 
elements $n$ share the same value for $i_n$ and thus share the same coloring convention, guaranteeing a common parity for $\text{VSG}(n)$.  From here, 
we proceed as in the proof of the previous theorem.
\end{proof}

\begin{cor}
$\RCA \vdash \HT^{\leq 3}_3 \rightarrow \ACA$. 
\end{cor}

\section{Open Questions}

Some of the open questions involve comparing bounded versions of
Hindman's Theorem with special cases of Ramsey's Theorem.
As usual, let $\RT^n_k$ denote Ramsey's Theorem for $k$-colorings of
$n$-element sets.  Thus, $\RT^n_k$ asserts that whenever the
$n$-element subsets of $\mathbb{N}$ are $k$-colored, there is an
infinite set $X \subseteq \mathbb{N}$ such that all $n$-element
subsets of $X$ have the same color.
 
We have provided some lower bounds on the strength and effective
content of some versions of Hindman's Theorem for bounded sums.
However, we do not know any upper bounds for the effective content and
strength of $\HT^{\leq n}_k$ for $n > 1, k > 1$ beyond those known
from \cite{BHS} for Hindman's Theorem itself.  In particular, we do not
know whether any of these bounded versions of Hindman's Theorem are
provable in $\ACA$, or whether any of them imply $\HT$.  We also do not
know whether $\HT^{\leq 2}_2$ implies $\ACA$ in $\RCA$, or whether
Ramsey's Theorem for $2$-coloring of pairs $\RT^2_2$ implies
$\HT^{\leq 2}_2$ in $\RCA$.

One might also consider the restriction of Hindman's Theorem to sums
of length exactly $n$.  Let $\HT^{=n}_k$ denote the assertion that for
each $k$-coloring $c : \mathbb{N} \rightarrow k$ there is an infinite
set $X \subseteq \mathbb{N}$ such that $\{\sum F | F \subseteq X
\text{ and } |F| = n\}$ is monochromatic.  It is clear that $\RT^n_k$
implies $\HT^{=n}_k$ in $\RCA$ for each $n, k \geq 1$, and indeed
$\HT^{=n}_k$ is just the restriction of $\RT^n_k$ to colorings $c$ of
$n$-element sets $F$ such that $c(F)$ depends only on $\sum F$.  It
follows from \cite{J}, Theorem 5.5, that each computable instance of
$\HT^{=n}_k$ has a $\Pi^0_n$ solution.  It is unknown whether this
result can be improved to $\Sigma^0_n$ or better.  It also remains
open for each $n, k \geq 2$ whether $\HT^{=n}_k$ implies $\RT^n_k$ in
$\RCA$.  We do not even know whether each computable instance of
$\HT^{=2}_2$ has a computable solution.


\begin{thebibliography}{99}

\bibitem{B} J. E. Baumgartner, A short proof of Hindman's theorem, J. Combinatorial Theory Ser. A
17 (1974), 384--386.

\bibitem{BHS} A. R. Blass, J. L. Hirst, and S. G. Simpson, Logical analysis of some theorems of combinatorics and topological dynamics,   pp. 125--156 in Logic and Combinatorics (Arcata, California, 1985), volume 65 of Contemporary Mathematics, American Mathematical Society, Providence R.I., 1987.

\bibitem{CLY} C. T. Chong, S. Lempp, and Y. Yang, On the role of the collection principle for $\Sigma_2$ formulas in second-order reverse mathematics, Proceedings of the American Mathematical Society,  138 (2010), 1093--1100.

\bibitem{C} W. W. Comfort, Ultrafilters: some old and some new results,   Bull. Amer. Math. Soc. 83 (1977), 417--455.

\bibitem{CJS} P. A. Cholak, C. G. Jockusch, and T. A. Slaman, On the strength of Ramsey's Theorem for pairs, J. Symbolic Logic 66 (2001), 1--55.

\bibitem{H} N. Hindman, Finite sums from sequences within cells of a partition of $\mathbb{N}$, J. Combinatorial Theory Ser. A, 17 (1974), 1--11.

\bibitem{HLS} N. Hindman, I. Leader, and D. Strauss, Open problems in partition regularity, Combinatorics, Probability, and Computing 12 (2003), 571--583.

\bibitem{J} C. Jockusch, Ramsey's theorem and recursion theory, J. Symbolic Logic 37 (1972), 268--280.
 
\bibitem{Si} S. Simpson, Subsystems of Second Order Arithmetic, Second Edition, Cambridge University Press, New York, NY, Association for Symbolic Logic, 2009.

\bibitem{So}  R. I. Soare, Recursively Enumerable Sets and Degrees, Perspectives in Mathematical Logic, Springer Verlag, Berlin, Heidelberg, 1987.
 
\bibitem{T}  H. Towsner, A simple proof and some difficult examples for Hindman's Theorem, Notre Dame Journal of Formal Logic, 53 (1) (2012),  53--65,

\end{thebibliography}
\end{document}